\documentclass[reqno,12pt]{amsart}        
\usepackage{amsfonts,amsmath,amssymb,amsthm,colordvi,mathrsfs,graphicx}
\usepackage{caption,subcaption,float}
\usepackage[utf8]{inputenc}
\usepackage{mathrsfs}
\usepackage{geometry}
\usepackage{enumerate}

\usepackage{amsmath, amssymb, amsthm, geometry, hyperref}

\usepackage{tabularx}
\usepackage{tabulary}
 
\hypersetup{
  colorlinks=true,
  linkcolor=blue,
  citecolor=blue,
  urlcolor=blue
}
\usepackage[all,knot,poly]{xy}
\usepackage{tikz-cd}
\usepackage{tikz}
\usepackage{verbatim}
\usetikzlibrary{arrows}
\usetikzlibrary{knots}
\tikzset{every path/.style={line width=0.4pt},every node/.style={transform shape,knot crossing,inner sep=1.5pt},>=triangle 60,text node/.style={rectangle,transform shape=false,black}}

 \usetikzlibrary{decorations.markings, arrows.meta}

\theoremstyle{plain}      
\newtheorem{thm}{Theorem}[section]     
\newtheorem{theorem}[thm]{\bf Theorem}     
     
\newtheorem{lemma}[thm]{\bf Lemma}     
\newtheorem{proposition}[thm]{\bf Proposition}

\theoremstyle{remark}      
\newtheorem{example}[thm]{Example} 
 
\newtheorem{remark}[thm]{Remark} 
     
\theoremstyle{definition}

\title[]{}
\subjclass[2020]{13D10, 14B10, 14B12, 14D15, 14B07}
\keywords{Infinitesimal deformations,  abstract rigidity, local rigidity, cones over projective varieties, isolated singularities}

\begin{document}

\title{On infinitesimal deformations of singular curves} 

\author{Mounir Nisse}
 
\address{Mounir Nisse\\
Department of Mathematics, Xiamen University Malaysia, Jalan Sunsuria, Bandar Sunsuria, 43900, Sepang, Selangor, Malaysia.
}
\email{mounir.nisse@gmail.com, mounir.nisse@xmu.edu.my}
\thanks{}

\title{On infinitesimal deformations of singular varieties III}

\maketitle


\begin{abstract}
We study the affine cone over a reducible nodal curve $X$
obtained by gluing three projective lines along three pairs of points to form a 
connected curve of arithmetic genus \(1\).  
We endow \(X\) with a line bundle \(L\) of multidegree \((4,3,3)\), and we show 
that \(L\) is very ample, giving an embedding 
\(\varphi_L:X\hookrightarrow \mathbb P^9\).  
We then analyze in detail the affine cone
\[
C(X)=\operatorname{Spec}\!\left( \bigoplus_{m\ge 0} H^0(X,L^{\otimes m}) \right)
\]
and determine its singular locus, which consists of three singular lines meeting 
at the vertex.  

Central to the deformation theory of the cone are the line bundles \(F_m\) that 
appear in Pinkham’s description of \(\mathbb G_m\)-equivariant deformations.  
We compute the cohomology of \(F_m\) for all \(m\in\mathbb Z\), proving in 
particular that the degree-zero bundle \(F_0=T_X\) satisfies 
\[
H^0(X,T_X)=0,\qquad H^1(X,T_X)=0,
\]
for a generic choice of gluing data.  
This implies \(T^1_{C(X)}(0)=0\), so the cone has no degree-zero infinitesimal 
equivariant deformations.  
We likewise determine the positive- and negative-weight components, and we use 
them to construct an explicit negatively graded smoothing of the cone.  
This example provides a fully computable instance of a non-isolated surface 
singularity whose deformation theory is governed entirely by the geometry of a 
reducible genus-one curve.
\end{abstract}

\section{Introduction}

Affine cones over projective varieties form a fundamental class of singularities 
in algebraic geometry.  
They appear naturally in birational geometry, in the study of degenerations, 
and in the compactification of moduli spaces via geometric invariant theory, 
and in analytic considerations concerning canonical metrics. 
Understanding their deformation theory is especially important: deforming the 
cone corresponds to deforming the underlying projective variety, but often in a 
more rigid and highly structured way due to the presence of a $\mathbb G_m$-action.

\vspace{0.2cm}

In gravitational physics, deformation theory offers a precise framework for
analyzing infinitesimal and finite perturbations of spacetime geometries,
including black hole solutions of Einstein’s equations.
Such deformations capture stability properties, quasi-normal modes, and
geometric variations of near-horizon structures, which are central to the
dynamics of event horizons.
From this perspective, deformation theory provides a mathematical tool for
probing how black hole geometries respond to physical perturbations and
external fields. This is an ongoing work  with Yen-Kheng Lim 
 on some applications in physics connected to the actual work ({\it c.f.} \cite{LN1-25}, and \cite{LN2-25}).

\vspace{0.3cm}

In this paper we study the cone over a particular nodal curve of genus one, 
embedded in projective space by a very ample line bundle of multidegree 
$(4,3,3)$.  
The curve is reducible, consists of three rational components, and has precisely 
three nodes.  
Its arithmetic genus is one, so it lies on the boundary of the moduli space of 
stable curves of genus one.  
As such, it provides a rich testing ground in which the geometry of nodal curves, 
their line bundles, and the deformation theory of the associated cones can all 
be computed explicitly.

\subsection*{\it Relation to moduli of stable curves}

The curve $X$ considered here represents a point in the boundary of the moduli 
stack $\overline{\mathcal M}_1$ of stable curves of genus one.  
Its dual graph contains a single cycle, and the moduli of its combinatorial type 
is zero-dimensional: up to isomorphism, there is essentially one such curve.  
Thus $X$ is \emph{rigid} as a nodal curve with fixed combinatorial structure.  
Nontrivial deformations of $X$ in the moduli space arise only by smoothing its 
nodes.  
This behavior is characteristic of stable curves of compact type or of curves 
whose automorphism group is finite.

The moduli of line bundles on $X$ forms a one-dimensional generalized Jacobian.  
This means that the pair $(X,L)$, with $L$ of fixed multidegree, varies in a 
one-dimensional family even though $X$ itself is combinatorially rigid.  
This interplay between the rigidity of the curve and the flexibility of its 
Picard group plays a central role in the deformation theory of the cone.

\noindent {\it The embedded curve and its cone.}
The line bundle $L$ of multidegree $(4,3,3)$ is very ample and defines an 
embedding of $X$ into projective space.  
The associated affine cone $C(X)$ is a normal surface with a 
one-dimensional singular locus consisting of three lines meeting at the vertex.  
This type of singularity is a natural generalization of the cones over smooth 
curves, but the reducible and nodal nature of $X$ introduces new deformation 
phenomena.

\begin{theorem}[Very ampleness]
The line bundle $L$ is very ample and embeds $X$ as a projectively normal curve.  
The associated cone $C(X)$ is normal and has a singular locus consisting of the 
vertex together with three one-dimensional components lying over the nodes of $X$.
\end{theorem}

\subsection*{\it Graded deformations of the cone}

A key tool in our study is Pinkham’s decomposition of the deformation space of a 
cone into graded pieces:
\[
T^1_{C(X)} = \bigoplus_{m\in\mathbb Z} T^1_{C(X)}(m),
\]
where each summand corresponds to deformations of weight $m$ under the natural 
$\mathbb G_m$-action.  
These graded pieces can be computed on the curve $X$ using cohomology of a 
family of line bundles $F_m$ that encode the action of the grading.

The cohomology of $F_m$ turns out to be particularly simple, depending only on 
the sign of $m$.


\vspace{0.1cm}

\begin{theorem}[Cohomology controlling deformations]
For all integers $m\neq 0$, the graded piece $T^1_{C(X)}(m)$ is governed by the 
cohomology of $F_m$:
\[
T^1_{C(X)}(m)\cong
\begin{cases}
H^1(X,F_m)^\vee & m>0,\\
H^0(X,F_m) & m<0.
\end{cases}
\]
Moreover,
\[
H^0(X,F_0)=0,\qquad H^1(X,F_0)=0,
\]
so the degree-zero graded component satisfies:
\(
T^1_{C(X)}(0)=0.
\)
\end{theorem}

This vanishing of the degree-zero graded piece is a strong rigidity phenomenon: 
it implies that the cone has no infinitesimal deformations that preserve the 
grading and do not change the weights.  
All deformations come strictly from positive or negative components of the 
graded deformation space.

\noindent {\it Smoothings of the cone.}
The negative-degree pieces \(T^1_{C(X)}(m)\) for \(m<0\) encode smoothings of the cone.
Using the explicit description of the cohomology of \(F_m\), we 
construct a one-parameter smoothing in 
which the entire singular locus of the cone is smoothed.

\begin{theorem}[Explicit \(m<0\) smoothing]\label{thm:intro-smoothing}
There exists a negatively graded one-parameter deformation
\[
\mathcal C^{(-1)}\to\Delta
\]
such that:
\begin{enumerate}
    \item \(\mathcal C^{(-1)}_0\cong C(X)\);
    \item for all \(t\neq 0\), the fiber \(\mathcal C^{(-1)}_t\) is smooth;
    \item the smoothing is induced by a nonzero class in \(T^1_{C(X)}(-1)\).
\end{enumerate}
\end{theorem}

Therefore, every singularity of the cone, including the vertex and the three 
one-dimensional branches, is simultaneously smoothed by this deformation.
The existence of such a smoothing shows that although the cone has a 
one-dimensional singular locus, it lies on the boundary of the moduli space of 
smooth surfaces and admits controlled degenerations governed entirely by the 
geometry of the curve $X$.

\vspace{0.1cm}

This example sits naturally at the intersection of several important moduli 
problems: (i) the moduli of stable curves of genus one; (ii) the moduli of polarized nodal curves $(X,L)$; (iii) deformation spaces of surface singularities with $\mathbb G_m$-actions; and (iv) smoothings of non-isolated singularities arising from cones.

In particular:
\begin{enumerate}
    \item The curve $X$ is rigid as a point in $\overline{\mathcal M}_1$, but 
          its pair $(X,L)$ moves in a $1$-dimensional family.
    \item The cone $C(X)$ has no degree-zero graded deformations, showing that 
          it is rigid with respect to deformations that preserve the grading.
    \item The smoothing of the cone corresponds, via Pinkham’s theory, to a 
          deformation of the pair $(X,L)$ that smooths the nodes of $X$.
\end{enumerate}

Namely,  the deformation theory of the cone is deeply linked to the moduli theory 
of the underlying curve, and the example provides a fully explicit illustration 
of how singularities of cones encode the geometry of their projective bases.
Moreover, this example illustrates how the geometry of a reducible nodal curve controls 
the deformation theory of the cone over its embedding.


\section{The curve \(X\) and its dual graph}

\begin{example}[{\it Embeddings of a Nodal Curve into Projective Space}]
Let
\(
\widetilde X = C_1 \sqcup C_2 \sqcup C_3,
\)
where each \(C_i \cong \mathbb P^1\) is a smooth projective line over an
algebraically closed field \(k\).
We choose distinct affine points
\(
a_1,a_2 \in C_1, \,\,
b_1,b_2,b_3 \in C_2\) and \(
c_1 \in C_3,
\)
and form a nodal curve \(X\) by identifying
\[
a_1 \sim c_1, \qquad
a_2 \sim b_2, \qquad
b_1 \sim b_3.
\]
We denote the normalization map by
\(
\nu : \widetilde X \longrightarrow X.
\)
By abuse of notation we still denote the images of the components of
\(\widetilde X\) in \(X\) by \(C_1,C_2,C_3\).

The curve \(X\) has three nodes:
\begin{enumerate}
    \item[(a)] one node joining \(C_1\) and \(C_3\), coming from \(a_1 \sim c_1\);
    \item[(b)] one node joining \(C_1\) and \(C_2\), coming from \(a_2 \sim b_2\);
    \item[(c)] one self-node on \(C_2\), coming from \(b_1 \sim b_3\).
\end{enumerate}
Hence, the components \(C_2\) and \(C_3\) do \emph{not} meet in \(X\); they
are disjoint as sub-curves.

The normalization \(\widetilde X\) is the disjoint union of three copies of
\(\mathbb P^1\), so its arithmetic genus is \(0\). Each node increases the
arithmetic genus by \(1\), but we lose \( (\#\text{components} - 1) \)
because we connect components into one curve. Explicitly,
\[
p_a(X) = \sum_i p_a(C_i) + \#\{\text{nodes}\} - \#\{\text{components}\} + 1
       = 0 + 3 - 3 + 1 = 1.
\]
Therefore, \(X\) is a (reducible) projective curve of arithmetic genus \(1\).

\noindent {\it Dual graph of \(X\).}
The dual graph \(\Gamma_X\) of a nodal curve \(X\) has:
\begin{itemize}
    \item[(i)] one vertex for each irreducible component \(C_i\);
    \item[(ii)] one edge for each node, joining the vertices corresponding to
    the components meeting at that node; a self-node corresponds to a loop.
\end{itemize}

\vspace{0.3cm}

\begin{theorem}\label{main-theorem}
Let \(L\) be a line bundle on \(X\) with multidegree
\(
\underline{\deg}(L) = (4,3,3),
\) where $X=(C_1 \sqcup C_2 \sqcup C_3)/\sim$ is the curve defined above,
and  let $C(X)$ be the cone over $X$. Then the following holds:
\[
{
\dim T^0(C(X))_m =
\begin{cases}
10m, & m\ge 1,\\[4pt]
0, & m\le 0,
\end{cases}}
\qquad
{
\dim T^1(C(X))_m =
\begin{cases}
0, & m\ge 0,\\[4pt]
-10m, & m\le -1.
\end{cases}}
\]
\begin{enumerate}
\item All nonzero \(T^1_m\) occur in negative degrees:  
\[
T^1_{-1},T^1_{-2},\dots
\]
These give all possible \(\mathbb{G}_m\)-equivariant smoothings of the cone vertex.

\item No positive-degree deformations exist.  
\item No equisingular deformations exist in degree \(0\).
\item The vertex of the cone is \emph{smoothable in infinitely many independent ways}.
\end{enumerate}
 \end{theorem}
 
\vspace{0.3cm}

In our case, we have three vertices \(v_1,v_2,v_3\) corresponding to
\(C_1,C_2,C_3\). The node \(a_2\sim b_2\) gives an edge between \(v_1\) and
\(v_2\); the node \(a_1\sim c_1\) gives an edge between \(v_1\) and \(v_3\);
and the self-node \(b_1\sim b_3\) on \(C_2\) corresponds to a loop at
\(v_2\). The dual graph \(\Gamma_X\) is:

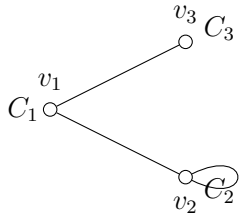
\vspace{0.3cm}
\begin{center}
\begin{tikzpicture}[scale=0.9,>=stealth]
    \node[circle,draw,inner sep=2pt,label=above:$v_1$] (v1) at (0,0) {};
    \node[circle,draw,inner sep=2pt,label=below:$v_2$] (v2) at (2,-1) {};
    \node[circle,draw,inner sep=2pt,label=above:$v_3$] (v3) at (2,1) {};
    \draw (v1) -- (v2); 
    \draw (v1) -- (v3); 
    \draw (v2) .. controls (3,-1.5) and (3, -0.5) .. (v2);
    \node at (-0.4,0) {$C_1$};
    \node at (2.5,-1.2) {$C_2$};
    \node at (2.5,1.2) {$C_3$};
\end{tikzpicture}
\end{center}

This graph has first Betti number \(1\), matching the arithmetic genus
\(p_a(X)=1\).

\vspace{0.2cm}

\noindent{\it Line bundles on \(X\) and multidegree.}
A line bundle \(L\) on the reducible curve \(X\) has a \emph{multidegree}
\(
\underline{\deg}(L) = (d_1,d_2,d_3),
\)
where
\(
d_i := \deg(L|_{C_i}) \in \mathbb Z.
\)
Since each \(C_i\) is isomorphic to \(\mathbb P^1\), we have
\(L|_{C_i} \cong \mathcal O_{\mathbb P^1}(d_i)\). The total degree of \(L\) is
\(
\deg(L) = d_1 + d_2 + d_3.
\)

The Picard group \(\mathrm{Pic}(X)\) is (non-canonically) an extension of a
torus by the group of multidegrees \(\mathbb Z^3\). For our purposes, it
suffices to track the multidegree and the fact that for each choice of
\((d_1,d_2,d_3)\) we can find line bundles on \(X\) with that multidegree.

\subsection{Global sections and Riemann--Roch.}

On each component \(C_i \cong \mathbb P^1\), we know
\[
h^0\bigl(\mathcal O_{C_i}(d_i)\bigr) =
\begin{cases}
d_i+1, & d_i \ge 0,\\
0, & d_i<0,
\end{cases}
\qquad
h^1\bigl(\mathcal O_{C_i}(d_i)\bigr) = 0
\quad\text{for } d_i \ge -1.
\]

Let \(L\) be a line bundle on \(X\) of multidegree \((d_1,d_2,d_3)\) with
all \(d_i \ge 0\). Then, global sections of \(L\) are triples of
sections
\[
(s_1,s_2,s_3)
\quad\text{with}\quad
s_i \in H^0\bigl(C_i, \mathcal O_{C_i}(d_i)\bigr),
\]
subject to compatibility conditions at the nodes (the sections must take
the same value on the pairs of points that are identified in \(X\)). Each
node produces \emph{at most one} linear condition on the space of such
triples.
Since we have three nodes, for \(d_i\) large enough (say \(d_i \ge 1\) for
all \(i\)), these conditions reduce the dimension by at most three, giving
\[
h^0(X,L) \ge \sum_{i=1}^3 (d_i + 1) - 3
= (d_1+d_2+d_3) = \deg(L),
\]
and one can show that for total degree \(\deg(L)\) sufficiently large,
\(L\) is non-special, i.e.\ \(h^1(X,L)=0\). Then Riemann--Roch for the
singular curve \(X\) (of arithmetic genus \(1\)) gives
\[
h^0(X,L) - h^1(X,L) = \deg(L) - p_a(X) + 1
= \deg(L).
\]
So, for large degree we indeed get
\(
h^0(X,L) = \deg(L),\) and so \(|L| \text{ maps }X\text{ to }\mathbb P^{\deg(L)-1}.
\)

\vspace{0.2cm}
  
For curves whose components are all isomorphic to \(\mathbb P^1\), we can
use the fact that on \(\mathbb P^1\), the line bundle
\(\mathcal O_{\mathbb P^1}(d)\) is very ample if and only if \(d\ge 2\), and
separates \(2\)-jets if \(d\ge 3\).

\vspace{0.4cm}

\begin{theorem}\label{thm:very_ample}
Let $L$ be a line bundle on $X$ with multidegree
\[
\underline{\deg}(L)=(d_1,d_2,d_3),
\]
where $d_i=\deg(L|_{C_i})$.
Suppose that
\[
d_i \ge 3 \quad\text{for each } i=1,2,3.
\]
Then $L$ is very ample: the complete linear series $|L|$
defines a closed immersion
\[
\varphi_L : X \hookrightarrow \mathbb P^{r}, \qquad r = h^0(X,L)-1.
\]
\end{theorem}

\vspace{0.2cm}

We will give a  proof, building on the description of
sections via the normalization and some elementary linear algebra on
$\mathbb P^1$.

\subsection{Description of sections via normalization}

Let $\tilde L := \nu^*L$ on $\widetilde X$. Since 
\[
\widetilde X = C_1 \sqcup C_2 \sqcup C_3,
\]
we have
\[
H^0(\widetilde X,\tilde L) \cong 
H^0(C_1,\tilde L|_{C_1}) \oplus
H^0(C_2,\tilde L|_{C_2}) \oplus
H^0(C_3,\tilde L|_{C_3}).
\]
Write
\[
L_i := \tilde L|_{C_i} \cong \mathcal O_{\mathbb P^1}(d_i),\qquad i=1,2,3,
\]
so $H^0(C_i,L_i)$ has dimension $d_i+1\ge 4$ by assumption $d_i\ge 3$.
Set
\[
V_i := H^0(C_i,L_i),\qquad i=1,2,3.
\]

Recall the normalization exact sequence:
\[
0 \longrightarrow \mathcal O_X
\longrightarrow \nu_*\mathcal O_{\widetilde X}
\longrightarrow \bigoplus_{p \in \text{nodes}(X)} k(p)
\longrightarrow 0.
\]
Tensoring with $L$ and taking global sections gives
\[
0 \longrightarrow H^0(X,L)
\longrightarrow H^0(\widetilde X,\tilde L)
\stackrel{\Phi}{\longrightarrow}
\bigoplus_{p \in \text{nodes}(X)} L\otimes k(p).
\]

There are three nodes $p_{13},p_{12},p_{22}$.
A section of $\tilde L$ is a triple
\[
(s_1,s_2,s_3) \in V_1 \oplus V_2 \oplus V_3,
\]
with $s_i \in H^0(C_i,L_i) = V_i$.
The map $\Phi$ records the \emph{differences} at preimages of each node:
\[
\Phi(s_1,s_2,s_3) = 
\bigl( s_1(a_1)-s_3(c_1),\;
       s_1(a_2)-s_2(b_2),\;
       s_2(b_1)-s_2(b_3) \bigr)
\in k^3.
\]
Thus
\[
H^0(X,L)
\cong
\ker \Phi
=
\left\{
(s_1,s_2,s_3)\in V_1\oplus V_2\oplus V_3
\;\middle|\;
\begin{array}{l}
s_1(a_1)=s_3(c_1),\\
s_1(a_2)=s_2(b_2),\\
s_2(b_1)=s_2(b_3)
\end{array}
\right\}.
\]

We will repeatedly use the fact that these conditions are \emph{linear}
in $(s_1,s_2,s_3)$, and hence solving them is a problem in linear algebra.

\subsection{Linear algebra on \texorpdfstring{$\mathbb P^1$}{P1}}

We collect some  facts about global sections of
$\mathcal O_{\mathbb P^1}(d)$.

\vspace{0.3cm}

\begin{lemma}[Independence of evaluation functionals]\label{lem:eval_indep}
Let $M \cong \mathcal O_{\mathbb P^1}(d)$ with $d\ge 1$, and let
$x_1,\dots,x_r$ be distinct points of $\mathbb P^1$ with $r\le d+1$.
Then the evaluation maps
\[
\mathrm{ev}_{x_j} : H^0(\mathbb P^1,M)\to k,\quad
s \mapsto s(x_j)
\]
are linearly independent linear functionals on $H^0(\mathbb P^1,M)$.
\end{lemma}

\begin{proof}
Choose an affine coordinate $t$ on $\mathbb A^1 \subset \mathbb P^1$
such that all $x_j$ lie in $\mathbb A^1 = \mathrm{Spec}\,k[t]$.
Then $M|_{\mathbb A^1}$ is trivial and
\[
H^0(\mathbb P^1,M) \cong k[t]_{\le d},
\]
the polynomials in $t$ of degree at most $d$.

Let $t_j\in k$ correspond to $x_j$. Evaluation at $x_j$ is the linear map
\[
\mathrm{ev}_{x_j}(f) = f(t_j),\qquad f\in k[t]_{\le d}.
\]
Suppose we have a linear relation
\[
\sum_{j=1}^r \lambda_j\,\mathrm{ev}_{x_j} = 0
\quad\text{as functionals on }k[t]_{\le d},
\]
with $\lambda_j \in k$.
That is,
\[
\sum_{j=1}^r \lambda_j f(t_j) = 0
\quad\text{for all } f \in k[t]_{\le d}.
\]

Since $r\le d+1$, the evaluation map
\[
k[t]_{\le d} \longrightarrow k^r,\quad
f \longmapsto (f(t_1),\dots,f(t_r))
\]
is surjective (Lagrange interpolation). In particular, for each fixed $j$,
there exists $f_j\in k[t]_{\le d}$ such that $f_j(t_i) = \delta_{ij}$.
Plugging $f_j$ into the linear relation yields $\lambda_j=0$.
Thus all $\lambda_j$ vanish, and the family
$\{\mathrm{ev}_{x_1},\dots,\mathrm{ev}_{x_r}\}$ is linearly independent.
\end{proof}

\vspace{0.3cm}

\begin{lemma}[Jets]\label{lem:jets}
Let $M \cong \mathcal O_{\mathbb P^1}(d)$ with $d\ge 2$, and let $x\in\mathbb P^1$.
Choose a local coordinate $u$ at $x$ on $\mathbb P^1$, so every $s\in H^0(\mathbb P^1,M)$
has a local expansion $s = a_0 + a_1u + (\text{higher order in }u)$.
Define
\[
\mathrm{ev}_x(s) := a_0,\qquad
\mathrm{ev}'_x(s) := a_1.
\]
Then:
\begin{enumerate}
  \item The two functionals $\mathrm{ev}_x,\mathrm{ev}'_x : H^0(\mathbb P^1,M)\to k$
        are linearly independent.
  \item If $y_1,\dots,y_r$ are distinct points in $\mathbb P^1\setminus\{x\}$
        and $r\le d-1$, then
        \[
        \mathrm{ev}_x,\ \mathrm{ev}'_x,\ \mathrm{ev}_{y_1},\dots,\mathrm{ev}_{y_r}
        \]
        are linearly independent linear functionals on $H^0(\mathbb P^1,M)$.
\end{enumerate}
\end{lemma}

\begin{proof}
Again choose an affine coordinate $t$ on $\mathbb A^1\subset\mathbb P^1$
such that $x$ and the $y_i$ all lie in $\mathbb A^1$.
Then $H^0(\mathbb P^1,M)\cong k[t]_{\le d}$.

We may assume $x$ corresponds to $t=0$ (by change of coordinate). Then
\[
\mathrm{ev}_x(f) = f(0),\qquad
\mathrm{ev}'_x(f) = f'(0).
\]
Write $f(t) = a_0 + a_1 t + \cdots + a_d t^d$.
Then
\[
\mathrm{ev}_x(f) = a_0,\qquad
\mathrm{ev}'_x(f) = a_1.
\]
Clearly these are independent linear forms in the coefficients $(a_0,\dots,a_d)$.

For (2), let $t_i$ correspond to $y_i$, so $\mathrm{ev}_{y_i}(f)=f(t_i)$.
Suppose we have a linear relation
\[
\lambda_0\,\mathrm{ev}_x + \lambda_1\,\mathrm{ev}'_x
  + \sum_{i=1}^r \mu_i\,\mathrm{ev}_{y_i} = 0
\]
on $k[t]_{\le d}$.
That is,
\[
\lambda_0 f(0) + \lambda_1 f'(0) + \sum_{i=1}^r \mu_i f(t_i) = 0
\quad\text{for all } f\in k[t]_{\le d}.
\]
We want to show all coefficients vanish.

First, choose $f$ of the form $f(t)=t^m$ for $m\ge 2$.
Then $f(0)=0$, $f'(0)=0$, and $f(t_i)=t_i^m$. We obtain
\[
\sum_{i=1}^r \mu_i t_i^m = 0
\quad\text{for all } m=2,\dots,d.
\]
This is a homogeneous linear system for the $\mu_i$.
Note that we have $d-1$ equations ($m=2,\dots,d$) and $r$ unknowns, with
$r\le d-1$ by hypothesis.
The Vandermonde matrix
\[
V = \bigl( t_i^m \bigr)_{\substack{2\le m\le d\\1\le i\le r}}
\]
has rank $r$ because the $t_i$ are distinct and we have at least $r$
distinct powers. Therefore, all $\mu_i=0$.
With $\mu_i=0$ for all $i$, the relation reduces to
\[
\lambda_0 f(0) + \lambda_1 f'(0) = 0
\quad\text{for all } f\in k[t]_{\le d}.
\]
Taking $f(t)=1$ yields $\lambda_0=0$, and $f(t)=t$ yields $\lambda_1=0$.
So all coefficients vanish, as required.
\end{proof}

\vspace{0.3cm}

\begin{remark}
In particular, for $d\ge 3$ we can impose independently:
\begin{itemize}
  \item[(1)] values at up to $4$ distinct points, or
  \item[(2)] a value and first derivative at one point, plus values at up to $2$
        further distinct points.
\end{itemize}
We will apply this repeatedly on each component $C_i\cong\mathbb P^1$.
\end{remark}

\subsection{Global generation}

We first recall (and slightly strengthen) the global generation result.

\vspace{0.3cm}

\begin{lemma}\label{lem:glob_gen}
Let $L$ be a line bundle on $X$ with multidegree $(d_1,d_2,d_3)$ and $d_i\ge 2$
for each $i$. Then $L$ is globally generated.
\end{lemma}

\begin{proof}

Let $x\in X$ be any point. We must find $(s_1,s_2,s_3)\in\ker\Phi$ such that
the value at $x$ is nonzero.

\smallskip

\emph{Case 1: $x$ is a smooth point of some component, say on $C_1$, and
$x\notin\{a_1,a_2\}$.}

Choose $s_1\in V_1$ with $s_1(x)\ne 0$ (possible since $L_1\simeq\mathcal O(d_1)$
with $d_1\ge 2$ is base-point free). Set $\alpha:=s_1(a_1)$, $\beta:=s_1(a_2)$.

Now choose $s_3\in V_3$ with $s_3(c_1)=\alpha$ (evaluation at $c_1$ is surjective),
and $s_2\in V_2$ with
\[
s_2(b_2) = \beta,\qquad
s_2(b_1) = s_2(b_3)
\]
(this is possible because the two linear functionals 
$\mathrm{ev}_{b_2}$ and $\mathrm{ev}_{b_1}-\mathrm{ev}_{b_3}$ on $V_2$
are independent by Lemma~\ref{lem:eval_indep} and $\dim V_2=d_2+1\ge 3$).

Then $(s_1,s_2,s_3)\in\ker\Phi$ and is nonzero at $x$.

\smallskip

\emph{Case 2: $x$ is one of the nodes $p_{13},p_{12},p_{22}$.}

One argues similarly: to make the section nonzero at the node, we ensure the
common value at the two preimages in $\widetilde X$ is nonzero, while solving the
other gluing conditions as above. The details are exactly as in the earlier lemma
and are omitted here.

\smallskip

All other smooth points on $C_2$ and $C_3$ are handled in the same way as Case~1.
Therefore, $L$ has no base points.
\end{proof}

In Theorem~\ref{thm:very_ample} we assume $d_i\ge 3$, so
Lemma~\ref{lem:glob_gen} applies and $L$ is globally generated.

\subsection{Criterion for very ampleness on curves}

We will use the standard criterion:

\vspace{0.3cm}

\begin{lemma}[Very ampleness]\label{lem:length2}
Let $X$ be a projective curve and $L$ a globally generated line bundle on $X$.
Then $L$ is very ample if and only if for every $0$-dimensional subscheme
$Z\subset X$ of length $2$, the restriction map
\[
H^0(X,L) \longrightarrow H^0(Z,L|_Z)
\]
is surjective.
\end{lemma}

\begin{proof}
See, e.g. Hartshorne \cite{H-77}.
We give a brief indication here.
Global generation of $L$ gives a morphism
\[
\varphi_L : X \to \mathbb P(H^0(X,L)^\vee).
\]
This map is a closed immersion if and only if: (i) it is injective on closed points, and (ii) it is an immersion (injective differential).
Both conditions can be checked infinitesimally on sub-schemes of length $2$:
\begin{itemize}
  \item[(i)] separation of two distinct points $x\ne y$ is equivalent to
        surjectivity for $Z=\{x,y\}$ (reduced length $2$ subscheme);
  \item[(ii)] injectivity of the differential at a point $x$ is equivalent to
        surjectivity for $Z=2x$ (the first infinitesimal neighbourhood).
\end{itemize}
On a nodal curve there are also length-$2$ subschemes supported at a node,
but the same principle applies since the tangent space at a node is two-dimensional;
surjectivity at all such $Z$ ensures that different branches are not identified
in the image and that the map is locally an embedding.

Conversely, if all $Z$ of length $2$ are embedded by $|L|$, then $\varphi_L$
is a closed immersion.

A completely rigorous argument can be found in the cited references.
\end{proof}

Thus, to prove Theorem~\ref{thm:very_ample}, it suffices to check the following:

\begin{enumerate}
  \item For any two distinct closed points $x,y\in X$, the map
        $H^0(X,L) \to L_x\oplus L_y$ is surjective.
  \item For any point $x\in X$ (smooth or node), the map
        $H^0(X,L)\to H^0(2x,L|_{2x})$ is surjective.
\end{enumerate}

We now carry this out using our explicit description of $H^0(X,L)$ and the
linear independence results on $\mathbb P^1$.

\subsection{Separation of distinct points}

Let $x\ne y$ be two distinct closed points of $X$.
We must show: given arbitrary $(\lambda,\mu)\in L_x\oplus L_y$, there exists
$s\in H^0(X,L)$ with $s(x)=\lambda$ and $s(y)=\mu$.
We will exploit the normalization. Let $\tilde x$ (resp.\ $\tilde y$) be the
preimage of $x$ (resp.\ $y$) in $\widetilde X$; if $x$ (or $y$) is a node, then
there are two preimages. We split into cases.

\medskip

\noindent\textbf{Case A: $x$ and $y$ are smooth points on (possibly different) components.}

Suppose first that $x,y$ are both smooth and not nodes.
For definiteness, we treat the most complicated case, when both lie on $C_2$,
and neither equals $b_1,b_2,b_3$; other configurations (one or both on $C_1$ or
$C_3$, or one equal to a node) are similar and easier.

So, assume
\(
x,y\in C_2\setminus\{b_1,b_2,b_3\},\) \, \( x\ne y,
\)
and $x,y$ correspond to distinct points $\tilde x,\tilde y\in C_2$ under $\nu$.
We want $(s_1,s_2,s_3)\in\ker\Phi$ such that
\[
s_2(\tilde x) = \lambda,\qquad s_2(\tilde y) = \mu.
\]

A convenient strategy is to take $s_1$ and $s_3$ identically zero and solve
everything on $C_2$. If we put
\[
s_1 = 0,\qquad s_3 = 0,
\]
the gluing conditions force
\[
0 = s_1(a_1) = s_3(c_1),\quad
0 = s_1(a_2) = s_2(b_2),\quad
s_2(b_1)=s_2(b_3).
\]
Thus, $s_2$ must satisfy
\begin{equation}\label{eq:s2_constraints}
s_2(b_2) = 0,\qquad
s_2(b_1)-s_2(b_3) = 0.
\end{equation}

\noindent We now work purely inside $V_2 = H^0(C_2,L_2)\cong H^0(\mathbb P^1,\mathcal O(d_2))$.
Define a linear map
\[
\Psi : V_2 \longrightarrow k^3,\qquad
s_2 \longmapsto \bigl( s_2(b_2),\; s_2(b_1)-s_2(b_3),\; s_2(\tilde x) \bigr).
\]
By Lemma~\ref{lem:eval_indep} and Lemma~\ref{lem:jets}, the functionals
\[
\mathrm{ev}_{b_2},\quad
\mathrm{ev}_{b_1}-\mathrm{ev}_{b_3},\quad
\mathrm{ev}_{\tilde x}
\]
are linearly independent on $V_2$ (they are linear combinations of evaluation
maps at the three distinct points $b_1,b_2,b_3$ and at $\tilde x$, and $d_2+1\ge 4$).
Hence, $\Psi$ has rank $3$, and its kernel has dimension
\[
\dim \ker\Psi = \dim V_2 - 3 = (d_2+1) - 3 = d_2-2 \ge 1
\]
since $d_2\ge 3$.
In particular, $\ker\Psi\neq 0$.

Next consider the evaluation functional at $\tilde y$ restricted to $\ker\Psi$.
If $\mathrm{ev}_{\tilde y}$ were zero on $\ker\Psi$, then
$\mathrm{ev}_{\tilde y}$ would lie in the span of the three functionals used
in defining $\Psi$, contradicting their linear independence together with
$\mathrm{ev}_{\tilde y}$ (again by Lemma~\ref{lem:eval_indep} since we have
at most $4$ distinct points here). Therefore $\mathrm{ev}_{\tilde y}$ is not
identically zero on $\ker\Psi$. Thus there exists a section
\[
t_2\in\ker\Psi \quad\text{with}\quad t_2(\tilde y)\ne 0.
\]

Now we proceed in two steps:

\smallskip

\emph{(i) Find $u_2\in V_2$ satisfying the gluing and the value at $\tilde x$.}

We first solve
\[
u_2(b_2)=0,\quad u_2(b_1)=u_2(b_3),\quad u_2(\tilde x)=\lambda.
\]
This is a system of $3$ independent linear conditions, hence soluble by the
surjectivity of $\Psi$. Choose one such $u_2$.

\smallskip

\emph{(ii) Adjust the value at $\tilde y$ using $\ker\Psi$.}

Now $u_2(\tilde y)$ may not equal $\mu$. But any element of $\ker\Psi$ can be
added without breaking the constraints~\eqref{eq:s2_constraints} or the
value at $\tilde x$, since $\Psi$ vanishes on $\ker\Psi$. Thus, we choose
$t_2\in\ker\Psi$ with $t_2(\tilde y)\ne 0$ as above and set
\[
s_2 := u_2 + c\,t_2,
\]
where $c\in k$ is chosen so that
\[
s_2(\tilde y) = u_2(\tilde y) + c\,t_2(\tilde y) = \mu.
\]
This is possible since $t_2(\tilde y)\ne 0$. Then $s_2$ satisfies
\[
s_2(b_2)=0,\quad s_2(b_1)=s_2(b_3),\quad
s_2(\tilde x)=\lambda,\quad s_2(\tilde y)=\mu.
\]

Finally, we set $s_1=s_3=0$. The triple $(s_1,s_2,s_3)$ lies in $\ker\Phi$
(because of the conditions on $s_2$), and has the desired values at $x$ and $y$.

\medskip

All other configurations of two smooth points (on $C_1$ or $C_3$, or on
different components) are treated in exactly similarly, with at most
$3$ independent linear conditions on the relevant $V_i$ and the same linear
algebra. We omit the repetitive details.

\medskip

\noindent\textbf{Case B: at least one of $x,y$ is a node.}

Suppose $x$ is a node and $y$ is arbitrary (smooth or node).
For concreteness, take $x=p_{13}=a_1\sim c_1$.
We want to specify the common value
\(
s_1(a_1) = s_3(c_1) = \lambda
\)
and the value at $y$.

If $y$ is a smooth point on some component, we may proceed as in Case~A,
choosing the relevant $s_i$ first on the components containing $x$ and $y$,
and then solving the remaining linear conditions on the remaining component.

As a sample, suppose $y\in C_2$ is smooth and not equal to $b_1,b_2,b_3$.
We can proceed as follows.
\begin{itemize}
  \item[(i)] Choose $s_1\in V_1$ with $s_1(a_1)=\lambda$ and some prescribed value
        $s_1(a_2)$ (for instance, we can leave $s_1(a_2)$ free for the moment).
  \item[(ii)] Choose $s_3\in V_3$ with $s_3(c_1)=\lambda$ (possible by evaluation surjectivity).
  \item[(iii)] Then $s_2\in V_2$ must satisfy
        \[
        s_2(b_2) = s_1(a_2),\qquad
        s_2(b_1)=s_2(b_3),\qquad
        s_2(y) = \mu.
        \]
        These are $3$ independent linear conditions on $V_2$ (by the same
        independence arguments as before), so there is a solution $s_2$.
\end{itemize}
The resulting triple $(s_1,s_2,s_3)$ lies in $\ker\Phi$ and satisfies
$s(x)=\lambda$, $s(y)=\mu$.

If $y$ is also a node, say $y=p_{12}$ or $p_{22}$, we argue similarly,
now prescribing common values on the two branches at that node and solving
linear conditions on the remaining component(s). The number of imposed
conditions on each $V_i$ never exceeds $3$ (values at at most three points,
or two values plus one equality of values), and $d_i\ge 3$ ensures enough
freedom to solve.

\medskip

This shows that \emph{for any two distinct points $x\ne y$ in $X$, the
restriction map $H^0(X,L)\to L_x\oplus L_y$ is surjective.}

\subsection{Separation of tangent directions (length-$2$ subschemes)}

It remains to show that, for each point $x\in X$, the restriction map
\[
H^0(X,L) \longrightarrow H^0(2x,L|_{2x})
\]
is surjective.

We again split by the type of point.

\medskip

\noindent\textbf{Case 1: $x$ a smooth point of $X$.}

Assume $x$ lies on $C_1$ and is not $a_1,a_2$; the other components are analogous.
Let $\tilde x\in C_1$ be its preimage in $\widetilde X$.

Locally at $x$, a length-$2$ subscheme $2x$ corresponds to specifying
the value and the first derivative of a section at $x$. Thus we must show:
given $(\alpha,\beta)\in k^2$, there exists
\[
(s_1,s_2,s_3)\in\ker\Phi
\]
such that, in a local coordinate $u$ at $\tilde x$ on $C_1$,
\[
s_1 = \alpha + \beta u + \text{(terms of order $\ge 2$)}.
\]

On the component $C_1\cong\mathbb P^1$ with $L_1\cong\mathcal O(d_1)$,
$d_1\ge 3$, the map
\[
V_1 = H^0(C_1,L_1) \longrightarrow
k^2,\quad
s_1 \mapsto \bigl( \mathrm{ev}_{\tilde x}(s_1),\mathrm{ev}'_{\tilde x}(s_1)\bigr)
\]
is surjective by Lemma~\ref{lem:jets}. Thus we can choose $s_1\in V_1$
with the desired value and derivative at $\tilde x$.
Let
\[
\alpha_1 := s_1(a_1),\qquad \alpha_2 := s_1(a_2)
\]
be its values at the two special points on $C_1$.

We now choose $s_3\in V_3$ with $s_3(c_1)=\alpha_1$, and then $s_2\in V_2$
with
\[
s_2(b_2)=\alpha_2,\qquad
s_2(b_1)=s_2(b_3),
\]
as in the proof of global generation. These impose only $1$ condition on $V_3$
and $2$ independent conditions on $V_2$, hence are soluble.

Then $(s_1,s_2,s_3)\in\ker\Phi$ and has the prescribed first-order jet at $x$.
Thus, $H^0(X,L)\to H^0(2x,L|_{2x})$ is surjective at every smooth point $x$.

\medskip

\noindent\textbf{Case 2: $x$ a node.}

Let $x=p_{13}=a_1\sim c_1$; the other nodes are analogous.
Near this node, $X$ has two smooth branches, one corresponding to $C_1$ and
one to $C_3$.
A length-$2$ subscheme supported at $x$ is of one of the following types:
\begin{enumerate}
  \item the union of the two reduced branches (length $1$ on each branch),
  \item a double point supported on the branch coming from $C_1$ (length $2$ on that branch),
  \item a double point supported on the branch coming from $C_3$.
\end{enumerate}
We must check surjectivity for each of these types.

\smallskip

\emph{Type (1):} $Z=\{a_1,c_1\}$ with the reduced scheme structure.
Then $H^0(Z,L|_Z)\cong L_{a_1}\oplus L_{c_1}\cong k^2$.
The restriction map factors through $H^0(\widetilde X,\tilde L)$, and
on $\widetilde X$ we know we can independently prescribe values at $a_1$ and $c_1$,
because $L_1$ and $L_3$ are globally generated and there is no constraint
coupling $s_1(a_1)$ and $s_3(c_1)$ \emph{before} imposing the node relation.

However, recall that sections of $L$ correspond to triples $(s_1,s_2,s_3)$
with
\[
s_1(a_1)=s_3(c_1).
\]
Thus, for a section of $L$, the values at $a_1$ and $c_1$ \emph{must} coincide.
So the image of the restriction map
\[
H^0(X,L) \to H^0(Z,L|_Z)\cong k\oplus k
\]
is the diagonal $\{(\lambda,\lambda)\}$, which is $1$-dimensional.

But here we must be careful: in Lemma~\ref{lem:length2}, the subscheme $Z$
of $X$ of length $2$ supported at a node \emph{as a subscheme of $X$} is
not the disjoint union of the two preimages on $\widetilde X$, but rather
a subscheme where these two ``branches'' meet at a single point. 
The relevant $Z$ of type (1) in the sense of Lemma~\ref{lem:length2} 
is \emph{the reduced node itself}, which has length $1$, not $2$.

Thus type (1) does \emph{not} occur as a length-$2$ subscheme of $X$.
The only length-$2$ subschemes supported at a node are the ``double points''
tangent to one branch or the other, i.e.\ types (2) and (3).
So we only need to check these double structures.

\smallskip

\emph{Type (2):} a double point supported on the $C_1$-branch.

Geometrically, this corresponds to specifying a first-order jet along the
$C_1$-branch at $a_1$, while imposing no condition along the $C_3$-branch.
Concretely, we must show that the map
\[
H^0(X,L)\to H^0(2a_1,L|_{2a_1})
\]
is surjective, where $2a_1$ is taken along the $C_1$-branch.

On $\widetilde X$, this jet is completely controlled by $s_1$ near $a_1$,
while $s_3$ only determines the value at $c_1$ (which must equal $s_1(a_1)$).

Thus, we must show: given $\alpha,\beta\in k$, there exists $(s_1,s_2,s_3)\in\ker\Phi$
such that in a local coordinate $u$ at $a_1$ on $C_1$,
\[
s_1 = \alpha + \beta u + \text{(higher order in $u$)}.
\]

This is entirely analogous to the smooth case on $C_1$.
We first pick $s_1\in V_1$ with prescribed value and derivative at $a_1$
(by Lemma~\ref{lem:jets}), and then extend to $(s_2,s_3)$ satisfying the
gluing conditions:
\[
s_3(c_1)=s_1(a_1)=\alpha,\qquad
s_2(b_2)=s_1(a_2),\qquad
s_2(b_1)=s_2(b_3).
\]
Again, these impose $1$ condition on $V_3$ and $2$ independent conditions on
$V_2$, hence they are solvable.

\smallskip

\emph{Type (3):} a double point supported on the $C_3$-branch.

This is completely symmetric to type (2), with $C_1$ and $C_3$ interchanged.
We choose $s_3$ with prescribed jet at $c_1$ and then solve for $s_1,s_2$
subject to the gluing conditions, using exactly the same linear algebra.

\medskip

\noindent\textbf{Other nodes.}

For the node $p_{12}=a_2\sim b_2$ (joining $C_1$ and $C_2$), the argument is exactly
the same: a length-$2$ subscheme supported at $p_{12}$ is a double point tangent
to the $C_1$-branch or to the $C_2$-branch. In either case, we prescribe a jet at
$a_2$ (resp.\ $b_2$) and solve the gluing conditions by choosing sections on the
other components $C_2$ and $C_3$ (resp.\ $C_1$ and $C_3$).

For the node $p_{22}=b_1\sim b_3$ (a self-node on $C_2$), the two branches both lie
on $C_2$. A double point tangent to one branch is locally the same as a double
point at $b_1$ or at $b_3$ on $C_2$. Thus we must show that given a jet at $b_1$
(or $b_3$), we can find $(s_1,s_2,s_3)\in\ker\Phi$ realizing it.

For a jet at $b_1$, we proceed as follows.
\begin{enumerate}
  \item Choose $s_2\in V_2$ with prescribed value and derivative at $b_1$
        (possible by Lemma~\ref{lem:jets}), and such that additionally
        $s_2(b_3)=s_2(b_1)$ (this last equality can be arranged because
        it is a \emph{single} linear condition, and we have $\dim V_2\ge 4$,
        so the space of sections with given jet at $b_1$ is $2$-dimensional
        and nonzero).
        More explicitly: first choose $s_2^0$ with the desired jet at $b_1$;
        then adjust $s_2^0$ by a section in the kernel of both
        $\mathrm{ev}_{b_1}$ and $\mathrm{ev}'_{b_1}$ so as to fix the value at $b_3$.
  \item With this $s_2$, set $\beta:=s_2(b_2)$.
        Then choose $s_1\in V_1$ with $s_1(a_2)=\beta$ (one condition),
        and $s_3\in V_3$ with $s_3(c_1) = s_1(a_1)$ (one condition).
\end{enumerate}
The gluing conditions
\[
s_1(a_1)=s_3(c_1),\quad
s_1(a_2)=s_2(b_2),\quad
s_2(b_1)=s_2(b_3)
\]
are then satisfied by construction, and we have the prescribed jet at $b_1$.
The case of a jet at $b_3$ is analogous.

\medskip

Thus, at every point $x\in X$, smooth or nodal, the map
\[
H^0(X,L)\to H^0(2x,L|_{2x})
\]
is surjective.

\begin{proof}[Proof of Theorem \ref{thm:very_ample}]
We have shown the following:
\begin{enumerate}
  \item $L$ is globally generated (Lemma~\ref{lem:glob_gen});
  \item for any two distinct points $x\ne y$ in $X$, the restriction map
        $H^0(X,L)\to L_x\oplus L_y$ is surjective;
  \item for any point $x\in X$, the restriction map
        $H^0(X,L)\to H^0(2x,L|_{2x})$ is surjective.
\end{enumerate}
By Lemma~\ref{lem:length2}, this implies that $L$ is very ample, i.e.\ the
morphism
\[
\varphi_L : X \to \mathbb P^{h^0(X,L)-1}
\]
is a closed immersion.

This completes the proof of Theorem~\ref{thm:very_ample}.
\end{proof}
\qed

\end{example}


\begin{remark}
The requirement \(d_i\ge 3\) for all \(i\) is a convenient sufficient
condition for very ampleness. There may exist very ample line bundles
with smaller multidegree, but \((d_1,d_2,d_3)\ge (3,3,3)\) (componentwise)
is an easy region where very ampleness is guaranteed.
\end{remark}

\section{Explicit examples of embeddings into \(\mathbb P^r\), \(r\ge 3\)}

We now give concrete examples of line bundles \(L\) on \(X\) such that
\(|L|\) embeds \(X\) into some projective space \(\mathbb P^r\) with
\(r\ge 3\). All these examples are based on Theorem~\ref{thm:very_ample}.

In each example we specify the multidegree \(\underline{\deg}(L)\), then
determine the dimension of the space of global sections and hence the
target projective space.

For line bundles of sufficiently large total degree (in particular in all
our examples), we have \(h^1(X,L)=0\), so by Riemann--Roch
\[
h^0(X,L) = \deg(L) = d_1 + d_2 + d_3
\quad\Rightarrow\quad
\varphi_L : X \hookrightarrow \mathbb P^{d_1+d_2+d_3 - 1}.
\]

\subsection*{Example 1: Multidegree \((3,3,3)\)}

Let \(L_1\) be any line bundle on \(X\) with multidegree
\[
\underline{\deg}(L_1) = (3,3,3),
\]
i.e.\ \(L_1|_{C_i} \cong \mathcal O_{C_i}(3)\) for each \(i=1,2,3\). Such a
line bundle exists because we may start with
\(\mathcal O_{C_1}(3)\boxplus \mathcal O_{C_2}(3)\boxplus \mathcal O_{C_3}(3)\)
on the normalization and perform a compatible gluing at the nodes.

We have \(d_i=3 \ge 3\) for all \(i\), so by
Theorem~\ref{thm:very_ample}, \(L_1\) is very ample. Its total degree is
\[
\deg(L_1) = 3+3+3 = 9,
\]
and for such large degree on a genus \(1\) curve we have \(h^1(X,L_1)=0\),
so
\[
h^0(X,L_1) = \deg(L_1) = 9.
\]
Therefore the complete linear series \(|L_1|\) defines an embedding
\[
\varphi_{L_1} : X \hookrightarrow \mathbb P^{8}.
\]

\subsection*{Example 2: Multidegree \((4,3,3)\)}

Let \(L_2\) be any line bundle on \(X\) with multidegree
\[
\underline{\deg}(L_2) = (4,3,3),
\]
i.e.
\[
L_2|_{C_1} \cong \mathcal O_{C_1}(4),\quad
L_2|_{C_2} \cong \mathcal O_{C_2}(3),\quad
L_2|_{C_3} \cong \mathcal O_{C_3}(3).
\]

Again, we have \(d_1=4\ge 3\), \(d_2=3\ge 3\), \(d_3=3\ge 3\), so
Theorem~\ref{thm:very_ample} implies that \(L_2\) is very ample.

The total degree is
\[
\deg(L_2) = 4+3+3 = 10,
\]
and for this degree we have \(h^1(X,L_2)=0\), hence
\[
h^0(X,L_2) = \deg(L_2) = 10.
\]
Therefore \(|L_2|\) defines an embedding
\[
\varphi_{L_2} : X \hookrightarrow \mathbb P^{9}.
\]

\subsection*{Example 3: Multidegree \((4,4,3)\)}

Let \(L_3\) be any line bundle on \(X\) with multidegree
\[
\underline{\deg}(L_3) = (4,4,3),
\]
i.e.
\[
L_3|_{C_1} \cong \mathcal O_{C_1}(4),\quad
L_3|_{C_2} \cong \mathcal O_{C_2}(4),\quad
L_3|_{C_3} \cong \mathcal O_{C_3}(3).
\]

We have \(d_1,d_2,d_3\ge 3\), so \(L_3\) is very ample by
Theorem~\ref{thm:very_ample}. The total degree is
\[
\deg(L_3) = 4+4+3 = 11,
\]
and again \(h^1(X,L_3)=0\), so
\[
h^0(X,L_3) = 11.
\]
Thus \(|L_3|\) defines an embedding
\[
\varphi_{L_3} : X \hookrightarrow \mathbb P^{10}.
\]

For the curve \(X\) described at the beginning, we have:

\begin{enumerate}
    \item A line bundle \(L\) with multidegree \((d_1,d_2,d_3)\) and
    \(d_i\ge 2\) is globally generated.
    \item A line bundle \(L\) with multidegree \((d_1,d_2,d_3)\) and
    \(d_i\ge 3\) for all \(i\) is very ample, and its complete linear
    series embeds \(X\) into \(\mathbb P^{\deg(L)-1}\).
    \item In particular, any line bundle with multidegree at least
    \((3,3,3)\) (componentwise) gives an embedding into some
    \(\mathbb P^r\) with \(r\ge 8\).
\end{enumerate}

Thus, the examples \(L_1,L_2,L_3\) above are concrete very ample line
bundles whose complete linear series embed \(X\) into projective spaces of
dimensions \(8,9,10\), respectively.


\vspace{0.3cm}

\section{The line bundle \(L(4,3,3)\) and the embedding into \(\mathbb P^9\)}

Let \(L\) be a line bundle on \(X\) with multidegree
\[
\underline{\deg}(L) = (4,3,3),
\]
i.e.
\[
L|_{C_1} \cong \mathcal O_{C_1}(4),\quad
L|_{C_2} \cong \mathcal O_{C_2}(3),\quad
L|_{C_3} \cong \mathcal O_{C_3}(3),
\]
where each \(C_i\cong\mathbb P^1\). Then
\[
\deg(L) = 4 + 3 + 3 = 10.
\]

Since each restriction \(L|_{C_i}\) has degree \(\ge 3\), one checks (by the
usual arguments on \(\mathbb P^1\) and compatibility at the nodes) that
\(L\) is very ample. Moreover, for a line bundle \(A\) on a curve of
arithmetic genus \(1\), the Riemann--Roch formula reads
\[
h^0(X,A) - h^1(X,A) = \deg(A).
\]
For \(L\) of degree \(10\) (and degree \(> 2p_a-2 = 0\)), we have
\(h^1(X,L)=0\). Hence
\[
h^0(X,L) = \deg(L) = 10,
\]
and the complete linear series \(|L|\) defines an embedding
\[
\varphi_L : X \hookrightarrow \mathbb P^{9}.
\]

\subsection{The affine cone \(C(X)\) and graded tangent cohomology}

Let
\[
R := \bigoplus_{m\ge 0} H^0(X,L^{\otimes m})
\]
be the (standard) section ring associated to \((X,L)\). The \emph{affine
cone} over \(X\) with respect to \(L\) is
\[
C(X) := \operatorname{Spec}(R),
\]
equipped with its natural \(\mathbb G_m\)-action coming from the grading
on \(R\). The vertex of the cone is the unique closed point corresponding
to the irrelevant ideal \(R_+ = \bigoplus_{m>0} R_m\).

\noindent The tangent cohomology of \(C(X)\), in the sense of 
Schlessinger, yields graded \(R\)-modules
\[
T^i(C(X)) = \bigoplus_{m\in\mathbb Z} T^i(C(X))_m,\quad i=0,1,\dots
\]
We are interested in the graded pieces
\(T^0(C(X))_m\) (infinitesimal automorphisms/derivations of degree \(m\))
and in the infinitesimal deformations
\(T^1(C(X))\), viewed through its graded pieces \(T^1(C(X))_m\).

There is a standard description (going back to M.~Schlessinger  \cite{S-68} and H.~C.~Pinkham \cite{P-74}) of the graded pieces of the tangent cohomology of the cone in terms of sheaf cohomology on \(X\). In particular, for a projective curve \(X\) with a very ample line bundle \(L\), and for the associated cone \(C(X)\), one has:

\begin{theorem}[Graded pieces]\label{Graded-pieces}
Let \(X\) be a projective curve and \(L\) very ample. Let
\(C(X)=\operatorname{Spec}R\) be the affine cone over \((X,L)\). Then, for
each integer \(m\neq 0\),
\begin{align*}
T^0(C(X))_m &\cong H^0\bigl(X,\,T_X\otimes L^{\otimes m}\bigr), \\
T^1(C(X))_m &\cong H^1\bigl(X,\,T_X\otimes L^{\otimes m}\bigr),
\end{align*}
as vector spaces over \(k\). For \(m=0\), there are short exact sequences
relating \(T^0(C(X))_0\) and \(T^1(C(X))_0\) to the intrinsic cohomology
\(H^i(X,T_X)\) and to embedded deformations of \(X\subset\mathbb P^9\). In
particular,
\[
H^i(X,T_X) \text{ occurs as a direct summand of } T^i(C(X))_0
\quad (i=0,1).
\]
\end{theorem}

We can find a complete proof with all the details in \cite{N2-25}. In this problem we restrict to the \emph{intrinsic} part coming from
\(H^i(X,T_X\otimes L^{\otimes m})\), and we assume that the embedding
\(X\subset\mathbb P^9\) is generic enough so that no extra linear
automorphisms or embedded deformations contribute in degree \(0\). Under
this assumption, and by Theorem \ref{Graded-pieces} we may identify:
\begin{align*}
T^0(C(X))_m &\cong H^0\bigl(X,\,T_X\otimes L^{\otimes m}\bigr)
\text{ for all } m\in\mathbb Z,\\
T^1(C(X))_m &\cong H^1\bigl(X,\,T_X\otimes L^{\otimes m}\bigr)
\text{ for all } m\in\mathbb Z.
\end{align*}
(If one wishes to be completely precise, one can treat \(m=0\) separately,
but for a generic curve \(X\) as above we will see that \(H^0(X,T_X)=0=
H^1(X,T_X)\), so this simplification is harmless.)

Therefore, the computation of the graded pieces of \(T^0\) and \(T^1\) reduces to
computing the cohomology of the line bundles
\[
F_m := T_X\otimes L^{\otimes m},\qquad m\in\mathbb Z.
\]

\subsection{The tangent sheaf \(T_X\) and the bundles \(F_m\)}

\subsubsection*{Degree of \(T_X\)}

The curve \(X\) is a reduced, irreducible, Gorenstein curve with only
nodal singularities and arithmetic genus \(p_a(X)=1\). The dualizing
sheaf \(\omega_X\) has degree
\[
\deg(\omega_X) = 2p_a(X)-2 = 0.
\]
For a Gorenstein curve of genus \(g\), the tangent sheaf is the (dual) line
bundle
\[
T_X \cong \omega_X^{-1},
\]
so it has degree
\[
\deg(T_X) = -\deg(\omega_X) = 0.
\]

\subsubsection{The bundles \(F_m = T_X\otimes L^{\otimes m}\)}

Define
\[
F_m := T_X\otimes L^{\otimes m},\qquad m\in\mathbb Z.
\]
Since \(T_X\) is a line bundle of degree \(0\) and \(L\) has degree
\(\deg(L)=10\), we have
\[
\deg(F_m) = \deg(T_X) + m\deg(L) = 0 + 10m = 10m.
\]

The (arithmetical) genus is \(p_a(X)=1\), and Riemann--Roch for a line
bundle \(F\) on \(X\) says
\[
h^0(X,F) - h^1(X,F) = \deg(F).
\]
Applied to \(F_m\), we obtain
\[
h^0(X,F_m) - h^1(X,F_m) = \deg(F_m) = 10m.
\]

\subsubsection{Vanishing of \(H^1(X,F_m)\) for \(\deg(F_m)>0\)}

Let \(F\) be a line bundle on a (reduced, Gorenstein) curve of genus \(1\).
If \(\deg(F)>0\), then \(H^1(X,F)=0\). This follows from Serre duality:
\[
H^1(X,F) \cong H^0\bigl(X,\omega_X\otimes F^{-1}\bigr)^\vee,
\]
and \(\deg(\omega_X\otimes F^{-1}) = \deg(\omega_X) - \deg(F)
= 0-\deg(F)<0\), so \(H^0(X,\omega_X\otimes F^{-1})=0\).

Applying this to \(F_m\), we obtain:

\begin{enumerate}
    \item For \(m>0\), \(\deg(F_m) = 10m>0\), so
    \[
    H^1(X,F_m)=0,\qquad
    h^0(X,F_m) = \deg(F_m) = 10m.
    \]
    \item For \(m<0\), \(\deg(F_m) = 10m<0\), so
    \[
    H^0(X,F_m)=0,\qquad
    h^1(X,F_m) = -\deg(F_m) = -10m.
    \]
    Indeed, since \(h^0-h^1=10m<0\) and \(h^0=0\), we get
    \(h^1=-10m>0\).
    \item For \(m=0\), we have \(F_0=T_X\) with \(\deg(F_0)=0\) and
    \[
    h^0(X,T_X) - h^1(X,T_X) = 0.
    \]
    In general, this does not fix \(h^0,h^1\) individually. However, for a
    \emph{generic} choice of the gluing points \(a_i,b_j,c_1\), the curve
    \(X\) has only a finite automorphism group, which implies
    \[
    H^0(X,T_X)=0.
    \]
    Then necessarily \(H^1(X,T_X)=0\) as well, since the Euler
    characteristic \(\chi(T_X)=0\).
\end{enumerate}

We adopt this genericity assumption (no nontrivial global vector fields),
so that, for \(m=0\),
\[
H^0(X,F_0)=H^0(X,T_X)=0,\qquad
H^1(X,F_0)=H^1(X,T_X)=0.
\]

\section{Computation of \(\dim T^0(C(X))_m\) and \(\dim T^1(C(X))_m\)}

Recall the identifications
\(
T^0(C(X))_m \cong H^0(X,F_m),\) \,\,and 
\(T^1(C(X))_m \cong H^1(X,F_m).
\)
We now summarize the dimensions in all degrees \(m\in\mathbb Z\).

\subsection{The space \(T^0(C(X))_m\)}

From the previous section we have:

\begin{enumerate}
    \item If \(m>0\): then \(\deg(F_m)=10m>0\), so
    \[
    h^1(X,F_m)=0,\quad h^0(X,F_m)=\deg(F_m)=10m.
    \]
    Therefore,
    \[
    \dim T^0(C(X))_m = 10m,\qquad m\ge 1.
    \]
    \item If \(m<0\): then \(\deg(F_m)=10m<0\), so
    \[
    H^0(X,F_m)=0,
    \]
    hence, 
    \[
    \dim T^0(C(X))_m = 0,\qquad m\le-1.
    \]
    \item If \(m=0\): under our genericity assumption,
    \[
    H^0(X,F_0) = H^0(X,T_X)=0,
    \]
    so intrinsically
    \[
    \dim H^0(X,F_0)=0.
    \]
    In the graded tangent cohomology of the cone, there is always an extra
    degree-zero derivation coming from the Euler vector field (the infinitesimal
    generator of the \(\mathbb G_m\)-action). If one includes this, one
    gets
    \[
    \dim T^0(C(X))_0 = 1
    \]
    (one dimensional space spanned by the Euler derivation). If one
    restricts to derivations coming from \(X\) itself, then
    \(\dim H^0(X,T_X)=0\) for the generic \(X\).
\end{enumerate}

Thus, for the intrinsic piece we have
\[
{
\dim T^0(C(X))_m =
\begin{cases}
10m, & m\ge 1,\\[4pt]
0, & m\le 0.
\end{cases}
}
\]
If one  include the Euler derivation at degree \(0\), then one
replaces the value \(0\) at \(m=0\) by \(1\).

\subsection{The space \(T^1(C(X))_m\)}

Similarly, we have
\[
T^1(C(X))_m \cong H^1(X,F_m),\qquad F_m=T_X\otimes L^{\otimes m}.
\]

\begin{enumerate}
    \item If \(m>0\), then \(\deg(F_m)=10m>0\). As before,
    \[
    H^1(X,F_m)=0,
    \]
    so
    \[
    \dim T^1(C(X))_m = 0,\qquad m\ge 1.
    \]
    \item If \(m<0\), then \(\deg(F_m)=10m<0\), and \(H^0(X,F_m)=0\). From
    Riemann--Roch,
    \[
    h^0(X,F_m)-h^1(X,F_m)=\deg(F_m)=10m,
    \]
    so
    \[
    0 - h^1(X,F_m) = 10m \quad\Longrightarrow\quad
    h^1(X,F_m) = -\,10m.
    \]
    Therefore,
    \[
    \dim T^1(C(X))_m = -10m,\qquad m\le -1.
    \]
    \item If \(m=0\), then \(F_0=T_X\) and, under our genericity assumption,
    \[
    H^0(X,T_X)=0,\quad H^1(X,T_X)=0.
    \]
    Thus,
    \[
    \dim T^1(C(X))_0 = 0
    \]
    for the intrinsic piece. (Depending on how one treats embedded deformations
    of \(X\subset\mathbb P^9\), \(T^1(C(X))_0\) can pick up a contribution
    from \(H^0(X,N_{X/\mathbb P^9})\); we are not including that here, since  we are treating
     the deformations coming from the cone structure itself.)
\end{enumerate}

So, the graded pieces of \(T^1\) are:

\[
{
\dim T^1(C(X))_m =
\begin{cases}
0, & m\ge 0,\\[4pt]
-10m, & m\le -1.
\end{cases}
}
\]

\subsection{Geometric interpretation.}

\begin{enumerate}
    \item All \emph{positive} graded pieces of \(T^1\) vanish: there are no
    nontrivial \(\mathbb G_m\)-equivariant deformations of the cone in positive
    degrees.
    \item All the \(\mathbb G_m\)-equivariant deformations of the cone are
    encoded in the \emph{negative} degrees \(m\le -1\), and their dimensions
    grow linearly with \(-m\).
    \item The graded pieces of \(T^0\) in positive degrees have dimension
    \(10m\), reflecting homogeneous vector fields of degree \(m\) on the cone,
    and in degree zero we have at least the Euler derivation.
\end{enumerate}

This completely determines the graded tangent cohomology of the cone
\(C(X)\) in terms of the line bundle \(L\) of multidegree \((4,3,3)\) that
embeds \(X\) into \(\mathbb P^9\).


\vspace{0.3cm}

\section{\(\mathbb{G}_m\)-Equivariant Deformations of the Cone}

We now explain \emph{how the signs of the degrees \(m\)} in the graded
modules \(T^1(C(X))_m\) correspond to \(\mathbb{G}_m\)-equivariant
deformations of the cone and how they affect the singularity at the vertex.

\subsection{The meaning of grading}
Since the cone
\(
C(X)=\mathrm{Spec}\Bigl( \bigoplus_{m\ge 0} H^0(X,L^{\otimes m}) \Bigr)
\)
has a natural \(\mathbb{G}_m\)-action given by rescaling the homogeneous
coordinates, its deformation theory decomposes into \emph{weight spaces}.

A deformation class in \(T^1(C(X))_m\) corresponds to a first-order deformation
\[
\widetilde{C(X)} \to C(X)
\]
together with a compatible \(\mathbb{G}_m\)-action of weight \(m\).

The degree \(m\) dictates how the deformation transforms under scaling:
\[
t\cdot \varepsilon = t^m \varepsilon.
\]

\subsection{Interpretation of the degrees}
The classical Pinkham theory (``deformations with \(\mathbb{G}_m\)-action'') says:

\begin{itemize}
\item \textbf{Negative weights \(m<0\):}  
These correspond to \emph{smoothing directions of the cone vertex}.  
Intuitively, negative weights ``add terms of smaller degree'' to the defining equations of the cone, which destroy the conical shape near the vertex.

\item \textbf{Weight \(m=0\):}  
These correspond to \emph{equisingular deformations that preserve the cone structure at the vertex}.  
No smoothing occurs; the vertex remains a singular point.  
(For our cone over the curve $X$ with $L$ of multidegree $(4,3,3)$, the intrinsic contribution to \(T^1_0\) is \(0\).)

\item \textbf{Positive weights \(m>0\):}  
These correspond to deformations that modify the embedding of \(X\) or the projective cone at infinity, but \emph{do not affect the affine cone near the vertex}.
For our cone, all such \(T^1_m=0\), so no such deformations exist.
\end{itemize}

\subsection{Consequences for our cone}

We computed:
\[
T^1(C(X))_m =
\begin{cases}
0, & m\ge 0,\\[4pt]
-10m, & m\le -1.
\end{cases}
\]

Therefore, we get:

\begin{itemize}
\item[(i)] \({T^1_m = 0 \text{ for } m\ge 0}\):

No equisingular or embedding-type equivariant deformations exist.
The cone has \emph{no non-smoothing} \(\mathbb{G}_m\)-equivariant deformations.

\item[(ii)] \({T^1_m \neq 0 \text{ only for } m<0}\):

All equivariant deformations are \emph{smoothing} of the vertex.
The dimension of the smoothing space in weight \(m\) is:
\(
\dim T^1_m = -10m.
\)

\item[(iii)] In particular, the \emph{miniversal} \(\mathbb{G}_m\)-equivariant deformation has a smoothing space of infinite dimension (one finite-dimensional graded piece for each negative integer).
\end{itemize}

\subsection{Geometric meaning}

A cone singularity is rigid \emph{as a cone} if and only if
\[
T^1_m = 0 \text{ for all } m < 0.
\]
For our curve:
\[
T^1_m \cong H^1(X,T_X\otimes L^m) \neq 0
\text{ for all } m<0.
\]
Hence, \textbf{The cone over \(X\subset\mathbb{P}^9\) is never rigid and always smoothable.}
The amount of smoothing in weight \(m\) grows linearly with \(-m\),
with slope equal to \(\deg(L)=10\).

\section{Moduli of $X$ and $(X,L)$}
\subsection{Moduli of the curve $X$ as a nodal curve of genus $1$}

We first discuss the moduli of $X$ as an abstract curve (disregarding $L$).

\subsubsection*{Combinatorial type and rigidity}

The dual graph of $X$ is fixed, and each irreducible component $C_i$ is $\mathbb P^1$.  
Let us examine the moduli of such curves up to isomorphism.

\begin{proposition}\label{prop:curve-rigid}
Up to isomorphism of curves, there is essentially a unique nodal curve $X$ of 
the above form: a union of three copies of $\mathbb P^1$ meeting in three nodes 
with the dual graph described above.  
Equivalently, the locus in the moduli stack of nodal curves of genus $1$ with 
this combinatorial type is $0$-dimensional.
\end{proposition}

\begin{proof}
We argue by normalizing each component and quotienting by automorphisms.

\medskip

\noindent
\emph{Component $C_1$.}  
The curve $C_1\simeq\mathbb P^1$ carries two special points $a_1,a_2$ (the 
nodes).  
The automorphism group $\operatorname{Aut}(\mathbb P^1)\cong\mathrm{PGL}_2$ acts 
$3$-transitively on $\mathbb P^1$, and hence transitively on ordered pairs of 
distinct points.  Thus, modulo automorphisms, there is no continuous parameter 
associated to the unordered pair $\{a_1,a_2\}$.  
We can choose coordinates so that
\[
a_1 = 0,\qquad a_2 = \infty.
\]

\noindent
\emph{Component $C_3$.}  
The component $C_3$ has a single special point $c_1$ coming from the node 
$a_1\sim c_1$.  
Again $\mathrm{PGL}_2$ acts transitively on $\mathbb P^1$, so we may use an 
automorphism of $C_3$ to put
\[
c_1 = 0.
\]

\noindent
\emph{Component $C_2$.}  
The component $C_2$ has three special points $b_1,b_2,b_3$.  
Up to automorphism of $\mathbb P^1$, any triple of distinct points is 
isomorphic to $\{0,1,\infty\}$.  
Indeed, the group $\mathrm{PGL}_2$ acts $3$-transitively on $\mathbb P^1$.  
Thus we may normalize
\[
b_1 = 0,\qquad b_2 = 1,\qquad b_3 = \infty.
\]

Under these normalizations, the gluing data
\[
a_1\sim c_1,\qquad a_2\sim b_2,\qquad b_1\sim b_3
\]
are completely prescribed and involve no continuous moduli.  
Hence any two such curves are isomorphic, so the moduli space of nodal curves 
with this combinatorial type consists of a single point (at the level of 
coarse moduli; the stack may still have a finite automorphism group).
\end{proof}

Thus, as an abstract nodal curve (with fixed combinatorial type), $X$ is rigid: 
there are no deformations that preserve the combinatorial structure and the 
number of nodes.  

On the other hand, if we allow \emph{smoothing} of the nodes, $X$ does deform 
inside the moduli stack $\overline{\mathcal M}_1$ of stable curves of genus $1$.  
Infinitesimally, deformations of the curve are governed by
\[
\mathrm{Ext}^1(\Omega_X,\mathcal O_X),
\]
which splits into \emph{locally trivial} deformations (measured by $H^1(X,T_X)$)
and deformations coming from smoothing the nodes (measured by the local 
$T^1$ at each node).  
In our situation, one checks that
\[
H^1(X,T_X) = 0
\]
for a generic choice of gluing points, so all deformations come from smoothing 
some subset of the three nodes.  
This produces a $3$-dimensional space of first-order deformations at the level 
of the deformation functor of $X$, which is then quotiented by the (finite) 
automorphism group when passing to the moduli stack.

\subsection{The moduli of the pair $(X,L)$}

We now fix the curve $X$ as above and consider line bundles $L$ on $X$ of fixed 
multidegree $(4,3,3)$.

\subsubsection*{The generalized Jacobian}

The Picard scheme $\mathrm{Pic}(X)$ of a connected projective curve of 
arithmetic genus $1$ has dimension $1$.  
It can be described as a \emph{generalized Jacobian}, which for a nodal genus 
one curve is an extension of an elliptic curve (in the smooth case) or a 
one-dimensional torus $\mathbb G_m$ (in the rational case) by a finite group.  
In our situation, $X$ is a genus one curve whose normalization is a union of 
three copies of $\mathbb P^1$, so its Jacobian is isomorphic to $\mathbb G_m$:
\[
\mathrm{Jac}(X) \cong \mathbb G_m.
\]

For each fixed multidegree $\underline d=(d_1,d_2,d_3)$ with $d_1+d_2+d_3$ fixed, 
the set of isomorphism classes of line bundles of multidegree $\underline d$ is 
a torsor under $\mathrm{Jac}(X)$.  
Thus the moduli space of pairs $(X,L)$, with $X$ as above and 
$\underline{\deg}(L)=(4,3,3)$, is $1$-dimensional:

\begin{proposition}\label{prop:pair-moduli}
Fix the curve $X$ described above.  
Then the set of isomorphism classes of line bundles of multidegree $(4,3,3)$ on 
$X$ is a principal homogeneous space under $\mathrm{Jac}(X)\cong\mathbb G_m$.  
In particular, the moduli space of pairs $(X,L)$ with $X$ fixed and 
$\underline{\deg}(L)=(4,3,3)$ is one-dimensional.
\end{proposition}

In the body of your work you fix a particular choice of $L$ in this torsor 
and prove that it is very ample and defines the embedding into $\mathbb P^9$.

\subsection{Moduli of the cone $C(X)$ as a singularity}

We now briefly discuss the local moduli of the cone $C(X)$ as a surface 
singularity.  
A precise notion of ``moduli space'' here requires some care: one usually works 
with the \emph{semi-universal deformation} (or miniversal deformation) of the 
singularity, whose base space is formally smooth of dimension equal to 
$\dim T^1_{C(X)}$, where
\[
T^1_{C(X)} = \mathrm{Ext}^1(\Omega_{C(X)},\mathcal O_{C(X)})
\]
is the space controlling first-order deformations.

Since $C(X)$ is a graded (affine) cone, its deformation theory admits a 
graded refinement:
\[
T^1_{C(X)} = \bigoplus_{m\in\mathbb Z} T^1_{C(X)}(m),
\]
where $T^1_{C(X)}(m)$ is the subspace of degree $m$ with respect to the 
$\mathbb G_m$-action on $C(X)$.
Pinkham's theory shows that these graded pieces can be expressed in terms of 
cohomology groups of certain line bundles $F_m$ on $X$.

\subsubsection*{The bundles $F_m$ and their cohomology}

For each integer $m$, there is a line bundle $F_m$ on $X$ such that:
\[
\deg(F_m)=10m,
\]
and
\[
T^1_{C(X)}(m) \cong
\begin{cases}
H^1(X,F_m)^\vee & m>0,\\[4pt]
H^0(X,F_m) & m<0,\\[4pt]
\text{a subquotient of } H^0(X,T_X)\oplus H^1(X,T_X)^\vee & m=0.
\end{cases}
\]
Using Riemann--Roch and Serre duality on the genus-one nodal curve $X$, and 
using the fact that the multidegree of $F_m$ is proportional to $m$, one 
obtains:

\begin{theorem}\label{thm:cohomology-Fm}
For the line bundles $F_m$ on $X$ arising from the graded structure of $C(X)$, 
we have:
\[
H^0(X,F_m)=0,\quad H^1(X,F_m)=-10m \quad\text{for } m<0,
\]
\[
H^1(X,F_m)=0,\quad H^0(X,F_m)=10m \quad\text{for } m>0.
\]
Moreover, for $m=0$ one has $F_0=T_X$ and
\[
h^0(X,T_X)-h^1(X,T_X)=0.
\]
For a generic choice of gluing points, the automorphism group 
$\mathrm{Aut}(X)$ is finite, and therefore
\[
H^0(X,T_X)=0,\qquad H^1(X,T_X)=0.
\]
\end{theorem}

The last statement (for $m=0$) implies in particular:

\begin{theorem}[Degree-zero graded deformations vanish]\label{thm:T1zero}
For a generic choice of the gluing points $a_i,b_j,c_1$, the degree-zero 
graded piece $T^1_{C(X)}(0)$ vanishes.  
Equivalently, the cone $C(X)$ admits no nontrivial $\mathbb G_m$-equivariant 
first-order deformations of degree $0$.
\end{theorem}

Thus, in the graded deformation theory, all nontrivial deformations of $C(X)$ 
occur in positive or negative degree.  

\begin{remark}
From the ``moduli'' point of view, the base space of the semi-universal 
\emph{graded} deformation of $C(X)$ is a formal scheme with tangent space
\[
\bigoplus_{m\ne 0} T^1_{C(X)}(m),
\]
and the vanishing of $T^1(0)$ means that there is no distinguished one-parameter 
family corresponding to regrading or to deformations preserving the $\mathbb G_m$-action 
in degree $0$.  
The negative-degree directions (in particular $m=-1$) give smoothings of the 
cone, while the positive-degree directions correspond to embedded deformations 
of the curve $X\subset\mathbb P^9$.
\end{remark}

\subsection{Summary of the moduli picture.} In this section and in the preceding sections we proved the following:

\begin{enumerate}[(a)]
  \item As an abstract nodal curve of genus $1$ with the given dual graph, $X$ 
        is rigid (Proposition~\ref{prop:curve-rigid}); all deformations in the 
        moduli of genus-one curves come from smoothing the nodes.
  \item For fixed $X$, the moduli of line bundles $L$ of multidegree $(4,3,3)$ 
        is one-dimensional, a torsor under the generalized Jacobian 
        $\mathrm{Jac}(X)\cong\mathbb G_m$ 
        (Proposition~\ref{prop:pair-moduli}).
  \item For the cone $C(X)$, the graded deformation space $T^1_{C(X)}$ has 
        no degree-zero component (Theorem~\ref{thm:T1zero}), and the 
        positive/negative degree pieces are computed cohomologically 
        (Theorem~\ref{thm:cohomology-Fm}).  
        The negative-degree pieces parametrize explicit $\mathbb G_m$-equivariant 
        smoothings of the cone.
\end{enumerate}

Therefore, the moduli of this cone, understood via its graded deformation theory, 
are tightly controlled by the geometry and cohomology of the underlying nodal 
curve $X$ and the line bundle $L$ of multidegree $(4,3,3)$.

Hence, we have proved  the following theorem:
\begin{theorem}\label{main-theorem}
Let \(L\) be a line bundle on \(X\) with multidegree
\(
\underline{\deg}(L) = (4,3,3),
\) where $X=$,
and  let $C(X)$ be the cone over $X$. Then the following holds:
\[
{
\dim T^0(C(X))_m =
\begin{cases}
10m, & m\ge 1,\\[4pt]
0, & m\le 0,
\end{cases}}
\qquad
{
\dim T^1(C(X))_m =
\begin{cases}
0, & m\ge 0,\\[4pt]
-10m, & m\le -1.
\end{cases}}
\]
\begin{enumerate}
\item All nonzero \(T^1_m\) occur in negative degrees:  
\[
T^1_{-1},T^1_{-2},\dots
\]
These give all possible \(\mathbb{G}_m\)-equivariant smoothings of the cone vertex.

\item No positive-degree deformations exist.  
\item No equisingular deformations exist in degree \(0\).
\item The vertex of the cone is \emph{smoothable in infinitely many independent ways}.
\end{enumerate}
 \end{theorem}


\end{document}